\documentclass[a4paper,oneside,11pt, reqno]{amsproc}
\usepackage{amsthm}
\usepackage{amsfonts}
\usepackage{amsmath}
\usepackage{amssymb}
\usepackage{mathrsfs}
\usepackage{epsfig}
\usepackage{graphicx}
\usepackage{epstopdf}
\usepackage{tikz-cd}
\usepackage{color}
\usepackage{ifthen}
\usepackage{float}
\usepackage{comment}

\makeatletter
\@namedef{subjclassname@2010}{%
\textup{2010} Mathematics Subject Classification}
\makeatother

\newtheorem{theorem}{Theorem}[section]
\newtheorem{proposition}[theorem]{Proposition}
\newtheorem{corollary}[theorem]{Corollary}
\newtheorem{lemma}[theorem]{Lemma}

\theoremstyle{remark}

\theoremstyle{definition}

\newcommand{\bq}{\begin{equation}}
\newcommand{\eq}{\end{equation}}
\newcommand{\beqn}{\begin{eqnarray*}}
\newcommand{\eeqn}{\end{eqnarray*}}
\newcommand{\beq}{\begin{eqnarray}}
\newcommand{\eeq}{\end{eqnarray}}

\newcommand{\rar}{\rightarrow}

\newcommand{\bc}{\begin{centre}}
\newcommand{\ec}{\end{centre}}

\newcommand{\ba}{\begin{array}}
\newcommand{\ea}{\end{array}}

\newcommand{\inp}[2]{\langle{#1},\,{#2} \rangle}

\renewcommand{\Delta}{{\nabla}}

\begin{document}
\title[Operator-valued Multishifts]{A short note on similarity of operator-valued multishifts}
\author[S. Ghara]{Soumitra Ghara}
\author[S. Kumar]{Surjit Kumar}
\author[S. Trivedi]{Shailesh Trivedi}
\address{Department of Mathematics\\
Indian Institute of Technology  Kharagpur, Midnapore-721302, India}
   \email{soumitra@maths.iitkgp.ac.in\\ ghara90@gmail.com}
\address{Department of Mathematics \\  Indian Institute of Technology Madras, Chennai  600036, India}
\email{surjit@iitm.ac.in}
\address{Department of Mathematics\\
Birla Institute of Technology and Science, Pilani 333031, India}
\email{shailesh.trivedi@pilani.bits-pilani.ac.in\\ shaileshtrivedi8@gmail.com}

\thanks{The work of the first author is supported by INSPIRE Faculty Fellowship (Ref. No. DST/INSPIRE/04/2021/002555). Support for the work of  the second author was provided in the form of the SERB-MATRICS grant (Ref. No. MTR/2022/000457) and NFIG grant (Ref. No. RF/21-22/1540/MA/NFIG/008991) of IIT Madras. 
The work of the third author was supported by Inspire Faculty Fellowship (Ref. No. DST/INSPIRE/04/2018/000338), OPERA (FR/SCM/03-Nov-2022/MATH) funded by BITS Pilani and SERB-SRG (SRG/2023/000641-G)}

\date{}

\begin{abstract}
A complete characterization of the similarity between two operator-valued multishifts with invertible operator weights is obtained purely in terms of operator weights. This generalizes several existing results of the unitary equivalence of two (multi)shifts. Further, we utilize the aforementioned similarity criteria to determine the similarity between two tuples of operators of multiplication by the coordinate functions on certain reproducing kernel Hilbert spaces determined by diagonal kernels.
\end{abstract}

\maketitle

\section{Introduction}

This article is another installment in the continuation of the study of operator-valued multishift initiated in \cite{GKT, GKT1}. In this short note, we focus our attention on the similarity of two operator-valued multishifts with invertible operator weights. We give a necessary and sufficient condition under which two operator-valued multishifts with invertible operator weights are similar. This condition is given purely in terms of the moments of the respective operator-valued multishifts. To explain the main result of this short note, we proceed by briefly recalling the preliminaries about the operator-valued multishifts. The reader is referred to \cite{GKT} and \cite{GKT1} for a detailed overview of the operator-valued multishifts. 

Let $\mathbb N$ denote the set of non-negative integers. Let $d$ be a positive integer and $H$ be a complex separable Hilbert space. We define $\ell^2_H(\mathbb N^d) :=\displaystyle\mathrel{\mathop\oplus_{\alpha \in \mathbb N^d}} H$ as the orthogonal direct sum of $\mbox{Card}(\mathbb N^d)$ many copies of $H$. Then $\ell^2_H(\mathbb N^d)$ is a Hilbert space equipped with the following inner product:
\beqn \inp{x}{y} := \sum_{\alpha \in \mathbb N^d} \inp{x_{\alpha}}{y_{\alpha}}_{_H},
\quad x = \displaystyle\mathrel{\mathop\oplus_{\alpha \in \mathbb N^d}} x_{\alpha},\
y=\displaystyle\mathrel{\mathop\oplus_{\alpha \in \mathbb N^d}} y_{\alpha} \in \ell^2_H(\mathbb N^d).\eeqn
Let $\big\{{A^{(j)}_{\alpha}} : \alpha \in  \mathbb N^d,\ j = 1, \ldots, d \big\}$ be a multisequence of bounded  invertible operators on $H$ with the following conditions: (i) $\sup_{\alpha \in \mathbb N^d} \|A^{(j)}_\alpha\| < \infty$ for all $j = 1, \ldots, d$ and (ii) $A^{(i)}_{\alpha+\varepsilon_j} A^{(j)}_{\alpha}= A^{(j)}_{\alpha+\varepsilon_i} A^{(i)}_{\alpha} \ \mbox{for all } \alpha \in \mathbb N^d$ and $i,j = 1, \ldots, d$, where $\varepsilon_j$ is the $d$-tuple in $\mathbb N^d$ in which $1$ is at the $j$-th place and $0$ elsewhere. Then an operator-valued multishift on $\ell^2_H(\mathbb N^d)$ is the commuting $d$-tuple $T = (T_1, \ldots, T_d)$ of bounded linear operators $T_1, \ldots, T_d$ on $\ell^2_H(\mathbb N^d)$ defined as
\beqn T_j(\displaystyle\mathrel{\mathop\oplus_{\alpha \in \mathbb N^d}} x_{\alpha}):= \displaystyle\mathrel{\mathop\oplus_{\alpha \in \mathbb N^d}}
A^{(j)}_{\alpha-\varepsilon_j}x_{\alpha-\varepsilon_j}, \quad \displaystyle\mathrel{\mathop\oplus_{\alpha \in \mathbb N^d}} x_{\alpha}
\in \ell^2_H(\mathbb N^d), \ j=1, \ldots, d.
\eeqn  
If for $\alpha = (\alpha_1, \ldots, \alpha_d) \in \mathbb N^d$, $\alpha_j = 0$, then we interpret $A^{(j)}_{\alpha-\varepsilon_j}$ as a zero operator and $x_{\alpha-\varepsilon_j}$ as a zero vector. By an $H$-valued formal power series we mean the series
$\displaystyle\sum_{\alpha \in \mathbb N^d} x_\alpha z^\alpha$, $x_\alpha \in H$, without regard to convergence
at any point $z \in \mathbb C^d$. For a given multisequence $\mathscr B = \big\{B_\alpha : \alpha \in \mathbb N^d\big\}$ of bounded invertible operators on $H$, the Hilbert space $\mathcal H^2(\mathscr B)$ of $H$-valued formal power series is defined as 
\beqn
\mathcal H^2(\mathscr B) := \left\{\sum_{\alpha \in \mathbb N^d} x_\alpha z^\alpha : \sum_{\alpha \in \mathbb N^d} \|B_\alpha x_\alpha\|^2 < \infty\right\}
\eeqn
 endowed with the following inner product:
  For $f(z) = \sum_{\alpha \in \mathbb N^d} x_\alpha z^\alpha$ and $g(z) = \sum_{\alpha \in \mathbb N^d} y_\alpha z^\alpha$ in $\mathcal H^2(\mathscr B)$,
 \beqn
 \inp{f(z)}{g(z)}_{\!_{\mathcal H^2(\mathscr B)}} := \sum_{\alpha \in \mathbb N^d}\inp{B_\alpha x_\alpha}{B_\alpha y_\alpha}_{\!_H}.
 \eeqn
Note that $H$-valued polynomials are dense in $\mathcal H^2(\mathscr B)$.  It follows from \cite[Theorem 5.4]{GKT1} that an operator-valued multishift on $\ell^2_H(\mathbb N^d)$ with invertible operator weights is unitarily equivalent to the commuting $d$-tuple $\mathscr M_z = (\mathscr M_{z_1}, \ldots, \mathscr M_{z_d})$ of (bounded) operators of multiplication by the coordinate function $z_j$ on $\mathcal H^2(\mathscr B)$, where $\mathscr B = \big\{B_\alpha : \alpha \in \mathbb N^d\big\}$ is given by \cite[Proposition 2.3]{GKT1}. In view of this equivalence, we work with the $\mathcal H^2(\mathscr B)$ model of an operator-valued multishift in the rest of this paper. The main result of this paper is stated as follows:

\begin{theorem}\label{similar-thm}
Let $H$ and $\tilde H$ be two complex separable Hilbert spaces and $\mathscr B = \{B_\alpha\}_{\alpha \in \mathbb N^d}$, $\tilde{\mathscr B} = \{\tilde B_\alpha\}_{\alpha \in \mathbb N^d}$ be two multisequences of bounded invertible operators on $H$ and $\tilde H$, respectively. Suppose that $\mathscr M_z = (\mathscr M_{z_1}, \ldots, \mathscr M_{z_d})$ is bounded on $\mathcal H^2(\mathscr B)$ and on $\mathcal H^2(\tilde{\mathscr B})$. Then $\mathscr M_z$ on $\mathcal H^2(\mathscr B)$ is similar to $\mathscr M_z$ on $\mathcal H^2(\tilde{\mathscr B})$ if and only if there exist a bounded invertible operator $C : \tilde H \to H$ and positive constants $m_1, m_2$ such that 
\beq\label{similar-ineq}
m_1 C^* B^*_\alpha B_\alpha C \leqslant {\tilde B_\alpha}^* \tilde B_\alpha \leqslant m_2 C^* B^*_\alpha B_\alpha C \quad \mbox{for all } \alpha \in \mathbb N^d.
\eeq
\end{theorem}

Even for the simplest scenario where $d=1$, the criteria for the similarity between two operator-valued unilateral weighted shifts weren't previously understood solely in terms of their moments. Instead, the similarity of these shifts was characterized by examining the intertwiner rather than relying on their moments, see \cite[Lemma 2.1]{L}.

The result of Theorem \ref{similar-thm} generalizes \cite[Lemma 2.2]{K} and \cite[Theorem $2'$(b)]{S}. The proof and the consequences of this theorem are presented in the next section. All the notaions and symbols used in the proof of Theorem \ref{similar-thm} are standard in the literature and have their usual meaning.

\section{Proof of Theorem \ref{similar-thm}}

We begin by setting the following notation.
\beq\label{H-alpha}
\mathcal H_\alpha := \{x z^\alpha : x \in H\}, \quad \alpha \in \mathbb N^d.
\eeq
 If $\alpha_j = 0$ for some $\alpha \in \mathbb N^d$, then we interpret $\mathcal H_{\alpha-\varepsilon_j}$ as $\{0\}$. Note that 
\beq\label{H-alpha-deco}
 \mathcal H^2(\mathscr B) = \displaystyle\mathrel{\mathop\oplus_{\alpha \in \mathbb N^d}} \mathcal H_\alpha.
\eeq 
The following well-known result (see \cite[Theorem 5.4]{GKT1}) will be used in the proof of the main result of this section.

\begin{lemma}\label{Mzj-invertible}
Let $H$ be a complex separable Hilbert space and $\mathscr B = \{B_\alpha\}_{\alpha \in \mathbb N^d}$ be a multisequence of bounded invertible operators on $H$. Let $\mathcal H^2(\mathscr B)$ be the Hilbert space of $H$-valued formal power series. Suppose that $\mathscr M_z = (\mathscr M_{z_1}, \ldots, \mathscr M_{z_d})$ is bounded on $\mathcal H^2(\mathscr B)$. 
Then the following statements are true:
\begin{enumerate}
\item[(i)] For all $x \in H$, $\alpha \in \mathbb N^d$ and  $ j = 1, \ldots, d,$ we have
\beq\label{Mz*}
\mathscr M^*_{z_j}(x z^\alpha) = (B_{\alpha-\varepsilon_j}^{*} B_{\alpha-\varepsilon_j})^{-1} B_\alpha^* B_\alpha x z^{\alpha-\varepsilon_j}, \ 
\eeq
where $B_{\alpha-\varepsilon_j}$ is interpreted as a zero operator if $\alpha_j = 0$ for some $\alpha \in \mathbb N^d$.
\item[(ii)] The linear map $P_{\mathcal H_{\alpha+\varepsilon_j}} \mathscr M_{z_j}|_{\mathcal H_\alpha} : \mathcal H_\alpha \rar \mathcal H_{\alpha+\varepsilon_j}$ is bounded and invertible for all $\alpha \in \mathbb N^d$ and $j = 1, \ldots, d$, where for a closed subspace $\mathcal E$ of $\mathcal H^2(\mathscr B)$, $P_{\mathcal E}$ denotes the orthogonal projection of $\mathcal H^2(\mathscr B)$ onto $\mathcal E$.
\end{enumerate}
\end{lemma}


We are now in a position to present the proof of the main result.

\begin{proof}[Proof of Theorem \ref{similar-thm}]
Suppose that $\mathscr M_z$ on $\mathcal H^2(\mathscr B)$ is similar to $\mathscr M_z$ on $\mathcal H^2(\tilde{\mathscr B})$. Let $X : \mathcal H^2(\mathscr B) \rar \mathcal H^2(\tilde{\mathscr B})$ be an invertible operator such that $X \mathscr M_{z_j} = \mathscr M_{z_j} X$ for all $j = 1, \ldots, d$. For each $\alpha \in \mathbb N^d$, let $\mathcal H_\alpha$ and $\tilde{\mathcal H}_\alpha$ be the respective subspaces of $\mathcal H^2(\mathscr B)$ and $\mathcal H^2(\tilde{\mathscr B})$ as defined by \eqref{H-alpha}. In view of the decomposition \eqref{H-alpha-deco}, we get
\beqn
P_{\tilde{\mathcal H}_\alpha}X \mathscr M_{z_j}|_{\mathcal H_\beta} = P_{\tilde{\mathcal H}_\alpha} \mathscr M_{z_j} X|_{\mathcal H_\beta} \mbox{ for all } \alpha, \beta \in \mathbb N^d \mbox{ and } j = 1, \ldots, d.
\eeqn
Consequently, it yields that 
\beq\label{alpha-gamma}
\sum_{\gamma \in \mathbb N^d} P_{\tilde{\mathcal H}_\alpha} X|_{\mathcal H_\gamma} P_{\mathcal H_\gamma} \mathscr M_{z_j}|_{\mathcal H_\beta} = \sum_{\gamma \in \mathbb N^d} P_{\tilde{\mathcal H}_\alpha} \mathscr M_{z_j}|_{\tilde{\mathcal H}_\gamma} P_{\tilde{\mathcal H}_\gamma} X|_{\mathcal H_\beta} \quad \mbox{for all } \alpha, \beta \in \mathbb N^d.
\eeq
Both the series in \eqref{alpha-gamma} converge in the strong operator topology. Further, note that for any $\alpha, \beta \in \mathbb N^d$, $P_{\mathcal H_\alpha}\mathscr M_{z_j}|_{\mathcal H_\beta} = 0$ if $\beta \neq \alpha - \varepsilon_j$. Using this in \eqref{alpha-gamma}, we get
\beq\label{X-alpha-alpha}
P_{\tilde{\mathcal H}_\alpha} X|_{\mathcal H_{\beta+\varepsilon_j}} P_{\mathcal H_{\beta+\varepsilon_j}} \mathscr M_{z_j}|_{\mathcal H_\beta} &=& P_{\tilde{\mathcal H}_\alpha} \mathscr M_{z_j}|_{\tilde{\mathcal H}_{\alpha-\varepsilon_j}} P_{\tilde{\mathcal H}_{\alpha-\varepsilon_j}} X|_{\mathcal H_\beta} \notag\\
&&\mbox{for all } \ \alpha, \beta \in \mathbb N^d \mbox{ and } j = 1, \ldots, d.
\eeq
Putting $\alpha = 0$ in \eqref{X-alpha-alpha} and using Lemma \ref{Mzj-invertible}, we get
\beq\label{X-0-alpha}
P_{\tilde{\mathcal H}_0} X|_{\mathcal H_{\beta+\varepsilon_j}} = 0 \ \mbox{ for all } \ \beta \in \mathbb N^d \ \mbox{ and } \ j = 1, \ldots, d.
\eeq
Further, putting $\alpha = \beta+\varepsilon_j$ in \eqref{X-alpha-alpha}, we get that
\beqn
P_{\tilde{\mathcal H}_{\beta + \varepsilon_j}} X|_{\mathcal H_{\beta+\varepsilon_j}} &=& P_{\tilde{\mathcal H}_{\beta+\varepsilon_j}} \mathscr M_{z_j}|_{\tilde{\mathcal H}_{\beta}} P_{\tilde{\mathcal H}_\beta}X|_{\mathcal H_\beta} (P_{\mathcal H_{\beta+\varepsilon_j}} \mathscr M_{z_j}|_{\mathcal H_\beta})^{-1} \notag\\ 
&&\mbox{ for all } \  \beta \in \mathbb N^d \mbox{ and } j = 1, \ldots, d.
\eeqn
Using recursion in the above expression, we get
\beq\label{recursion}
P_{\tilde{\mathcal H}_\alpha} X|_{\mathcal H_\alpha} = P_{\tilde{\mathcal H}_\alpha} \mathscr M^\alpha_z|_{\tilde{\mathcal H}_0} P_{\tilde{\mathcal H}_0} X|_{\mathcal H_0} (P_{\mathcal H_\alpha} \mathscr M^\alpha_z|_{\mathcal H_0})^{-1} \ \mbox{ for all } \ \alpha \in \mathbb N^d.
\eeq
Note that $P_{\tilde{\mathcal H}_0} X|_{\mathcal H_0} : \mathcal H_0 \rar \tilde{\mathcal H}_0$ is an invertible operator. Indeed, 
\beqn
I_{\tilde{\mathcal H}_0} = P_{\tilde{\mathcal H}_0} X X^{-1}|_{\tilde{\mathcal H}_0} = \sum_{\alpha \in \mathbb N^d} P_{\tilde{\mathcal H}_0} X|_{\mathcal H_\alpha} P_{\mathcal H_\alpha} X^{-1}|_{\tilde{\mathcal H}_0} \overset{\eqref{X-0-alpha}}= P_{\tilde{\mathcal H}_0} X|_{\mathcal H_0} P_{\mathcal H_0}X^{-1}|_{\tilde{\mathcal H}_0}.
\eeqn
Letting $C := P_{\mathcal H_0}X^{-1}|_{\tilde{\mathcal H}_0}$, we deduce from \eqref{recursion} that for all $\alpha \in \mathbb N^d$,
\beq\label{first-ineq1}
P_{\tilde{\mathcal H}_0} \mathscr M^{*\alpha}_z \mathscr M^\alpha_z|_{\tilde{\mathcal H}_0} = C^* (P_{\mathcal H_\alpha} \mathscr M^\alpha_z|_{\mathcal H_0})^* (P_{\tilde{\mathcal H}_\alpha} X|_{\mathcal H_\alpha})^* (P_{\tilde{\mathcal H}_\alpha} X|_{\mathcal H_\alpha}) (P_{\mathcal H_\alpha} \mathscr M^\alpha_z|_{\mathcal H_0}) C.
\eeq
Through repeated applications of \eqref{Mz*}, we obtain that
$$P_{\mathcal H_0} \mathscr M^{*\alpha}_z \mathscr M^\alpha_z|_{\mathcal H_0} = B^*_\alpha B_\alpha \mbox{ for all } \alpha \in \mathbb N^d.$$ 
Also, observe that $(P_{\tilde{\mathcal H}_\alpha} X|_{\mathcal H_\alpha})^* (P_{\tilde{\mathcal H}_\alpha} X|_{\mathcal H_\alpha}) \leqslant \|X\|^2 I.$
Substituting these values in \eqref{first-ineq1}, we get 
\beq\label{first-ineq}
{\tilde B_\alpha}^* \tilde B_\alpha \leqslant \|X\|^2 C^* B_\alpha^* B_\alpha C \mbox{ for all } \alpha \in \mathbb N^d.
\eeq
Similar arguments for $X^{-1}$ yield that
\beq\label{second-ineq}
\frac{1}{\|X^{-1}\|^2} C^* B_\alpha^* B_\alpha C \leqslant {\tilde B_\alpha}^* \tilde B_\alpha \quad \mbox{for all } \alpha \in \mathbb N^d.
\eeq
Combining \eqref{first-ineq} and \eqref{second-ineq}, we get
\beqn
\frac{1}{\|X^{-1}\|^2} C^* B_\alpha^* B_\alpha C \leqslant {\tilde B_\alpha}^* \tilde B_\alpha \leqslant \|X\|^2 C^* B_\alpha^* B_\alpha C \mbox{ for all } \alpha \in \mathbb N^d,
\eeqn
establishing \eqref{similar-ineq}.

Conversely, suppose that there exist a bounded invertible operator $C$ on $H$ and positive constants $m_1, m_2$ such that \eqref{similar-ineq} holds. Define $X : \mathcal H^2(\mathscr B) \rar \mathcal H^2(\tilde{\mathscr B})$ as
\beqn
X\Big(\sum_{\alpha \in \mathbb N^d} x_\alpha z^\alpha\Big) = \sum_{\alpha \in \mathbb N^d} (C^{-1}x_\alpha) z^\alpha, \quad \sum_{\alpha \in \mathbb N^d} x_\alpha z^\alpha \in \mathcal H^2(\mathscr B).
\eeqn
Note that $X$ is a block diagonal operator with respect to the decomposition $\mathcal H^2(\mathscr B) = \displaystyle\mathrel{\mathop\oplus_{\alpha \in \mathbb N^d}} \mathcal H_\alpha$. Let $X = \displaystyle\mathrel{\mathop\oplus_{\alpha \in \mathbb N^d}} X_\alpha$, where $X_\alpha : \mathcal H_\alpha \rar \tilde{\mathcal H}_\alpha$ is given by $X_\alpha (xz^\alpha) = (C^{-1}x) z^\alpha$ for all $\alpha \in \mathbb N^d$ and $x \in H$. Then \eqref{similar-ineq} implies that $\sqrt{m_1}\|xz^\alpha\| \leqslant \|X_\alpha (xz^\alpha)\| \leqslant \sqrt{m_2}\|xz^\alpha\|$ for all $\alpha \in \mathbb N^d$ and $x \in H$. Thus $X$ is a bounded invertible linear map. Further, the verification of the fact that $X \mathscr M_{z_j} = \mathscr M_{z_j} X$ for all $j = 1, \ldots, d$, is obvious. This completes the proof.
\end{proof}

The proof of the preceding theorem shows that the similarity between the multiplication operator tuples $\mathscr M_z$ on $\mathcal H^2(\mathscr B)$ and on $\mathcal H^2(\tilde{\mathscr B})$ can be characterized through an intertwining diagonal invertible operator (cf. \cite[Lemma 2.1]{L}).  

\begin{corollary}
Let $H$, $\tilde H$ be two complex separable Hilbert spaces and $\mathscr B = \{B_\alpha\}_{\alpha \in \mathbb N^d}$, $\tilde{\mathscr B} = \{\tilde B_\alpha\}_{\alpha \in \mathbb N^d}$ be two multisequences of invertible operators on $H$ and $\tilde H$, respectively. Suppose that $\mathscr M_z = (\mathscr M_{z_1}, \ldots, \mathscr M_{z_d})$ is bounded on $\mathcal H^2(\mathscr B)$ and on $\mathcal H^2(\tilde{\mathscr B})$. Then $\mathscr M_z$ on $\mathcal H^2(\mathscr B)$ is similar to $\mathscr M_z$ on $\mathcal H^2(\tilde{\mathscr B})$ if and only if there exists a bounded invertible operator $C : \tilde H \to H$ such that the map $X : \mathcal H^2(\mathscr B) \rar \mathcal H^2(\tilde{\mathscr B})$, given by $X f(z) = C^{-1} f(z)$ for all $f \in \mathcal H^2(\mathscr B)$, is bounded and invertible. 
\end{corollary}

\begin{proof}
If $\mathscr M_z$ on $\mathcal H^2(\mathscr B)$ is similar to $\mathscr M_z$ on $\mathcal H^2(\tilde{\mathscr B})$, then by the proof of the preceding theorem we can construct a bounded invertible operator $X : \mathcal H^2(\mathscr B) \rar \mathcal H^2(\tilde{\mathscr B})$ given by $X f(z) = C^{-1} f(z)$ for all $f \in \mathcal H^2(\mathscr B)$. Clearly, $X$ satisfies $X \mathscr M_{z_j} = \mathscr M_{z_j} X$ for all $j = 1, \ldots, d$. The converse part is straight forward.
\end{proof}

The following corollary generalizes as well as provides an alternate verification of \cite[Theorem 5.8]{GKT1}. It also recovers all the criteria of unitary equivalence established in \cite{L} and \cite{S}.

\begin{corollary}\label{cor-2.3}
Let $H$, $\tilde H$ be two complex separable Hilbert spaces and $\mathscr B = \{B_\alpha\}_{\alpha \in \mathbb N^d}$, $\tilde{\mathscr B} = \{\tilde B_\alpha\}_{\alpha \in \mathbb N^d}$ be two multisequences of invertible operators on $H$ and $\tilde H$, respectively. Suppose that $\mathscr M_z = (\mathscr M_{z_1}, \ldots, \mathscr M_{z_d})$ is bounded on $\mathcal H^2(\mathscr B)$ and on $\mathcal H^2(\tilde{\mathscr B})$. Then the following statements are equivalent:
\begin{enumerate}
\item[(i)] $\mathscr M_z$ on $\mathcal H^2(\mathscr B)$ is unitarily equivalent to $\mathscr M_z$ on $\mathcal H^2(\tilde{\mathscr B})$.
\item[(ii)] There exists a unitary operator $V : \tilde H \to H$ such that 
\beqn
 {\tilde B_\alpha}^* \tilde B_\alpha = V^* B^*_\alpha B_\alpha V \quad \mbox{for all } \alpha \in \mathbb N^d.
\eeqn
\item[(iii)] There exists a unitary operator $V : \tilde H \to H$ such that the map $X : \mathcal H^2(\mathscr B) \rar \mathcal H^2(\tilde{\mathscr B})$, given by $X f(z) = V^* f(z)$ for all $f \in \mathcal H^2(\mathscr B)$, is unitary.
\end{enumerate}
\end{corollary}

\begin{proof}
The equivalence of (i) and (ii) is immediate from the proof of Theorem \ref{similar-thm}, whereas that of (i) and (iii) goes along the lines of the proof of the preceding corollary.
\end{proof}

\section{Applications to reproducing kernel Hilbert space}

In this section, we apply Theorem \ref{similar-thm} to deduce the similarity between two commuting tuples of multiplication operators on certain reproducing kernel Hilbert spaces with diagonal kernels in which the polynomials are dense. We proceed by recalling a few basic facts about such reproducing kernel Hilbert spaces (for more details, see \cite{S, CS, PR}).

Let $\Omega$ be a connected open subset of $\mathbb C^d$ containing 0 and $E$ be a complex separable Hilbert space. Let $\mathcal H(\kappa)$ be the reproducing kernel Hilbert space of $E$-valued holomorphic functions defined on $\Omega$, where the reproducing kernel $\kappa$ is of the form
\beqn
\kappa(z,w) = \sum_{\alpha\in \mathbb N^d} C_\alpha z^\alpha \bar{w}^\alpha, \ z, w \in \Omega;
\eeqn
with $C_\alpha$ being bounded positive invertible operators on $E$. By a {\it diagonal kernel} we mean a reproducing kernel $\kappa$ of the above form. Moreover, if $\mathscr M_{z_j}$ is bounded on $\mathcal H(\kappa)$ for all $j = 1, \ldots, d$, then $\mathscr M_z = (\mathscr M_{z_1}, \ldots, \mathscr M_{z_d})$ on $\mathcal H(\kappa)$ is unitarily equivalent to $\mathscr M_z$ on  
$\mathcal H^2(\mathscr B)$, where $\mathscr B = \{B_\alpha\}_{\alpha \in \mathbb N^d}$ with $B_\alpha = C_\alpha^{-1/2}$ for all $\alpha \in \mathbb N^d$.

The following result shows that perturbing finitely many operator coefficients of $\kappa$ does not affect the similarity of $\mathscr M_z$.

\begin{proposition}
Let $E$ be a complex separable Hilbert space and $\Omega$  be a connected open subset of $\mathbb C^d$ containing 0. Let $\mathcal H(\kappa)$ and $\mathcal H(\tilde\kappa)$ be the  reproducing kernel Hilbert spaces of $E$-valued holomorphic functions defined on $\Omega$, where $\kappa$ and $\tilde\kappa$ are given by
\beqn   
\kappa(z,w) &=& \sum_{\alpha\in \mathbb N^d} C_\alpha z^\alpha \bar{w}^\alpha, \quad z, w \in \Omega;\\ 
\tilde\kappa(z,w) &=& \sum_{\underset{|\alpha| \leqslant n}{\alpha \in \mathbb N^d}} D_\alpha z^\alpha \bar{w}^\alpha + \sum_{\underset{|\alpha| \geqslant n+1}{\alpha \in \mathbb N^d}} C_\alpha z^\alpha \bar{w}^\alpha, \quad z, w \in \Omega;\eeqn 
with $C_\alpha$, $D_\alpha$ being positive invertible operators on $E$. If $\mathscr M_z$ is bounded on $\mathcal H(\kappa)$, then the following statements are true:
\begin{itemize}
\item[(i)] $\mathscr M_z$ is bounded on $\mathcal H(\tilde\kappa)$.
\item[(ii)] $\mathscr M_z$ on $\mathcal H(\kappa)$ is similar to $\mathscr M_z$ on $\mathcal H(\tilde\kappa)$.
\end{itemize}
\end{proposition} 

\begin{proof}
To see (i), let $j \in \{1, \ldots, d\}$. Note that $\mathscr M_{z_j}$ is bounded on $\mathcal H(\kappa)$ if and only if there exists a positive constant $\delta$ such that $(\delta - z_j \bar{w}_j) \kappa(z,w)$ is a positive definite kernel (refer to \cite[Theorem 6.28]{PR}). It follows from \cite[Lemma 4.1(c)]{CS} that $\delta C_\alpha - C_{\alpha-\varepsilon_j} \geqslant 0$ for all $\alpha \in \mathbb N^d$. Consequently, we get that $\mathscr M_{z_j}$ is bounded on $\mathcal H(\kappa)$ if and only if $\sup_{\alpha \in \mathbb N^d} \|C_\alpha^{-1/2} C_{\alpha-\varepsilon_j} C_\alpha^{-1/2}\| < \infty$. Here we follow the convention that $C_{\alpha-\varepsilon_j}=0$ if $\alpha_j=0$. This establishes that the boundedness of $\mathscr M_z$ on $\mathcal H(\kappa)$ is equivalent to that of $\mathscr M_z$ on $\mathcal H(\tilde\kappa)$.  

For the proof of (ii), suppose that $\mathscr M_z$ is bounded on $\mathcal H(\kappa)$. Note that $\mathscr M_z$ on $\mathcal H(\kappa)$ is unitarily equivalent to $\mathscr M_z$ on $\mathcal H^2(\mathscr B)$ and $\mathscr M_z$ on $\mathcal H(\tilde\kappa)$ is unitarily equivalent to $\mathscr M_z$ on $\mathcal H^2(\tilde{\mathscr B})$, where $\mathscr B = \{B_\alpha\}_{\alpha \in \mathbb N^d}$, $\tilde{\mathscr B} = \{\tilde B_\alpha\}_{\alpha \in \mathbb N^d}$ with 
\[ B_\alpha =C_\alpha^{-1/2} \mbox{ for all } \alpha \in \mathbb N^d \ \mbox{~and~}\; \tilde B_\alpha  =\begin{cases} D_\alpha^{-1/2} & \mbox{~if~}  |\alpha| \leqslant n,\\
C_\alpha^{-1/2} &  \mbox{~if~} |\alpha| \geqslant n+1.
\end{cases}\]
By choosing  
\beqn
c_1 &=& \min\left\{\frac{1}{\|C_\alpha^{-1/2} D_\alpha C_\alpha^{-1/2}\|} : |\alpha| \leqslant n\right\} \mbox{ and}\\
c_2 &=& \max\Big\{\|C_\alpha^{1/2} D_\alpha^{-1} C_\alpha^{1/2}\| : |\alpha| \leqslant n\Big\},
\eeqn
we obtain that
\beqn
c_1 C_\alpha^{-1} \leqslant D_\alpha^{-1} \leqslant c_2 C_\alpha^{-1} \mbox{ for all } \alpha \in \mathbb N^d \mbox{ with } |\alpha| \leqslant n.
\eeqn
Now if we take $m_1 = \min\{1, c_1\}$, $m_2 = \max\{1, c_2\}$ and $C=I$, then Theorem \ref{similar-thm} establishes the similarity between $\mathscr M_z$  on $\mathcal H^2(\mathscr B)$ and $\mathscr M_z$  on $\mathcal H^2(\tilde{\mathscr B})$. This completes the proof.
\end{proof}

In the preceding proposition, $\mathscr M_z$ on $\mathcal H(\kappa)$ is not unitarily equivalent to $\mathscr M_z$ on $\mathcal H(\tilde\kappa)$, in general (see Corollary \ref{cor-2.3}). The following proposition provides another class of reproducing kernel Hilbert spaces in which the similarity between two tuples of multiplication operators can be deduced through Theorem \ref{similar-thm}.

\begin{proposition}\label{prop-2}
Let $\lambda,\ \mu,\ \tilde\lambda,\ \tilde\mu$ be $($strictly$)$ positive real numbers not necessarily distinct. Let $\mathcal H(\kappa)$ and $\mathcal H(\tilde\kappa)$ be the reproducing kernel Hilbert spaces of $\mathbb C^2$-valued holomorphic functions defined on the open unit $($euclidean$)$ ball $\mathbb B_d$ of $\mathbb C^d$, where $\kappa$ and $\tilde\kappa$ are given by
\beqn
\kappa(z,w) &=& \left(\begin{matrix}
\frac{1}{(1-\inp{z}{w})^\lambda} & 0\\
&\vspace{-.3cm}\\
0 & \frac{1}{(1-\inp{z}{w})^\mu} \end{matrix}\right),\quad z, w \in \mathbb B_d;\\
\tilde\kappa(z,w) &=& \left(\begin{matrix}
\frac{1}{(1-\inp{z}{w})^{\tilde\lambda}} & 0\\
&\vspace{-.3cm}\\
0 & \frac{1}{(1-\inp{z}{w})^{\tilde\mu}} \end{matrix}\right),\quad z, w \in \mathbb B_d.
\eeqn 
Then $\mathscr M_z$ on $\mathcal H(\kappa)$ is similar to $\mathscr M_z$ on $\mathcal H(\tilde\kappa)$ if and only if $\{\lambda, \mu\} = \{\tilde\lambda, \tilde\mu\}$.
\end{proposition} 

\begin{proof}
Routine calculations show $\mathscr M_z$ is bounded on $\mathcal H(\kappa)$ and on $\mathcal H(\tilde\kappa)$. Further,  
\beqn
\kappa(z,w) &=& \sum_{\alpha \in \mathbb N^d} \frac{1}{\alpha!}\left(\begin{matrix}
(\lambda)_{|\alpha|} & 0\\
&\vspace{-.3cm}\\
0 & (\mu)_{|\alpha|} \end{matrix}\right) z^\alpha \bar{w}^\alpha ,\quad z, w \in \mathbb B_d;\\
\tilde\kappa(z,w) &=& \sum_{\alpha \in \mathbb N^d} \frac{1}{\alpha!}\left(\begin{matrix}
(\tilde\lambda)_{|\alpha|} & 0\\
&\vspace{-.3cm}\\
0 & (\tilde\mu)_{|\alpha|} \end{matrix}\right) z^\alpha \bar{w}^\alpha,\quad z, w \in \mathbb B_d,
\eeqn
where for $x>0$ and $n \in \mathbb N$, $(x)_n := \frac{\Gamma(x+n)}{\Gamma(x)}$. Thus $\mathscr M_z$ on $\mathcal H(\kappa)$ is unitarily equivalent to $\mathscr M_z$ on $\mathcal H^2(\mathscr B)$ and $\mathscr M_z$ on $\mathcal H(\tilde\kappa)$ is unitarily equivalent to $\mathscr M_z$ on $\mathcal H^2(\tilde{\mathscr B})$, where $\mathscr B = \{B_\alpha\}_{\alpha \in \mathbb N^d}$, $\tilde{\mathscr B} = \{\tilde B_\alpha\}_{\alpha \in \mathbb N^d}$ with 
\beqn
(B^*_\alpha B_\alpha)^{-1} = \frac{1}{\alpha!}\left(\begin{matrix}
(\lambda)_{|\alpha|} & 0\\
&\vspace{-.3cm}\\
0 & (\mu)_{|\alpha|} \end{matrix}\right)\ \mbox{ and }\ (\tilde B^*_\alpha \tilde B_\alpha)^{-1} = \frac{1}{\alpha!}\left(\begin{matrix}
(\tilde\lambda)_{|\alpha|} & 0\\
&\vspace{-.3cm}\\
0 & (\tilde\mu)_{|\alpha|} \end{matrix}\right) 
\eeqn  
for all $\alpha \in \mathbb N^d$.

Suppose that $\{\lambda, \mu\} = \{\tilde\lambda, \tilde\mu\}$. If $\lambda = \tilde\lambda$ and $\mu = \tilde\mu$, then clearly $\mathscr M_z$ on $\mathcal H(\kappa)$ is similar to $\mathscr M_z$ on $\mathcal H(\tilde\kappa)$. Let $\lambda = \tilde\mu$ and $\mu = \tilde\lambda$. Then take $m_1 = m_2 = 1$ and $C = \left(\begin{matrix}
0 & 1\\
&\vspace{-.3cm}\\
1 & 0 \end{matrix}\right)$. It now immediately follows from Theorem \ref{similar-thm} that $\mathscr M_z$ on $\mathcal H(\kappa)$ is similar to $\mathscr M_z$ on $\mathcal H(\tilde\kappa)$.

Conversely, suppose that $\mathscr M_z$ on $\mathcal H(\kappa)$ is similar to $\mathscr M_z$ on $\mathcal H(\tilde\kappa)$. Hence, $\mathscr M_z$ on $\mathcal H^2(\mathscr B)$ is similar to $\mathscr M_z$ on $\mathcal H^2(\tilde{\mathscr B})$. Therefore, by Theorem \ref{similar-thm}, there exists an invertible matrix $C = \left(\begin{matrix}
c_{11} & c_{12}\\
&\vspace{-.3cm}\\
c_{21} & c_{22} \end{matrix}\right)$ and positive constants $m_1$, $m_2$ such that for all $n \in \mathbb N$,
\beq\label{eq-12}
m_1 \left(\begin{matrix}
(\lambda)_n & 0\\
&\vspace{-.3cm}\\
0 & (\mu)_n \end{matrix}\right) \leqslant C \left(\begin{matrix}
(\tilde\lambda)_n & 0\\
&\vspace{-.3cm}\\
0 & (\tilde\mu)_n \end{matrix}\right) C^* \leqslant m_2 \left(\begin{matrix}
(\lambda)_n & 0\\
&\vspace{-.3cm}\\
0 & (\mu)_n \end{matrix}\right).
\eeq
The last inequality of \eqref{eq-12} gives that 
\begin{align} 
 \label{eq-13}
\left. \arraycolsep=1pt\def\arraystretch{1.4}
 \begin{array}{cc}
 &  |c_{11}|^2 (\tilde\lambda)_n + |c_{12}|^2 (\tilde\mu)_n \leqslant m_2 (\lambda)_n,  \\ 
 &  |c_{21}|^2 (\tilde\lambda)_n + |c_{22}|^2 (\tilde\mu)_n \leqslant m_2 (\mu)_n,  
 \end{array}
\right\}
n \in \mathbb N.
\end{align}
From the first inequality of \eqref{eq-12}, we get
\beqn
m_1 C^{-1}\left(\begin{matrix}
(\lambda)_n & 0\\
&\vspace{-.3cm}\\
0 & (\mu)_n \end{matrix}\right) C^{*-1} \leqslant \left(\begin{matrix}
(\tilde\lambda)_n & 0\\
&\vspace{-.3cm}\\
0 & (\tilde\mu)_n \end{matrix}\right) \mbox{ for all } n \in \mathbb N.
\eeqn
This gives that 
\begin{align} 
 \label{eq-14}
\left. \arraycolsep=1pt\def\arraystretch{1.4}
 \begin{array}{cc}
 & \frac{m_1}{|\det(C)|^2}\big( |c_{22}|^2 (\lambda)_n + |c_{12}|^2 (\mu)_n\big) \leqslant (\tilde\lambda)_n,  \\ 
 & \frac{m_1}{|\det(C)|^2}\big( |c_{21}|^2 (\lambda)_n + |c_{11}|^2 (\mu)_n\big) \leqslant (\tilde\mu)_n,  
 \end{array}
\right\}
n \in \mathbb N.
\end{align}
Now we consider the following two cases:\\
{\bf Case I:} Both $c_{11}$ and $c_{22}$ are non-zero.\\
Then by \eqref{eq-13} and \eqref{eq-14}, we get
\begin{align*} 
\left. \arraycolsep=1pt\def\arraystretch{1.4}
 \begin{array}{cc} 
& \frac{m_1}{|\det(C)|^2} |c_{22}|^2 (\lambda)_n \leqslant (\tilde\lambda)_n \leqslant \frac{m_2}{|c_{11}|^2} (\lambda)_n, \\ & \frac{m_1}{|\det(C)|^2} |c_{11}|^2 (\mu)_n \leqslant (\tilde\mu)_n \leqslant \frac{m_2}{|c_{22}|^2} (\mu)_n, 
\end{array}
\right\}
n \in \mathbb N.
\end{align*}
Since $\frac{(\tilde{\lambda})_n}{(\lambda)_n}\sim n^{\tilde{\lambda}-\lambda}$ as $n\to \infty$ (see \cite[p. 257 (6.1.46)]{AS}), it follows from above that $\lambda = \tilde\lambda$. By a similar argument, we get $\mu = \tilde \mu$.\\
{\bf Case II:} One of $c_{11}$ and $c_{22}$ is zero.\\
In this case, both $c_{12}$ and $c_{21}$ must be non-zero as $C$ is invertible. Hence, again by \eqref{eq-13} and \eqref{eq-14}, we obtain 
\begin{align*} 
\left. \arraycolsep=1pt\def\arraystretch{1.4}
 \begin{array}{cc} 
& \frac{m_1}{|\det(C)|^2} |c_{21}|^2 (\lambda)_n \leqslant (\tilde\mu)_n \leqslant \frac{m_2}{|c_{12}|^2} (\lambda)_n, \\ & \frac{m_1}{|\det(C)|^2} |c_{12}|^2 (\mu)_n \leqslant (\tilde\lambda)_n \leqslant \frac{m_2}{|c_{21}|^2} (\mu)_n, 
\end{array}
\right\}
n \in \mathbb N.
\end{align*}
Again, using the argument as in the Case I, we get that $\tilde\lambda = \mu$ and $\tilde\mu = \lambda$. This completes the proof.
\end{proof}

We conclude this paper by characterizing the similarity between two $d$-tuples of $\mathcal U(d)$-homogeneous multiplication operators acting on the reproducing kernel Hilbert spaces determined by $n \times n$ matrix-valued diagonal kernels defined on the open unit ball in $\mathbb C^d$.

Let $\mathcal U(d)$ denote the group of $d \times d$ unitary matrices and $E$ be a complex separable Hilbert space. Recall that a commuting $d$-tuple $\mathscr M_z = (\mathscr M_{z_1}, \ldots, \mathscr M_{z_d})$ of multiplication operators on a reproducing kernel Hilbert space $\mathcal H(\kappa)$ of $E$-valued holomorphic functions defined on the open unit ball $\mathbb B_d$ of $\mathbb C^d$ is said to be $\mathcal U(d)$-{\it homogeneous} (also known as {\it spherical}, see \cite{CY}) if for every $u = (u_{ij}) \in \mathcal U(d)$ there is a unitary operator $\Gamma(u)$ on $\mathcal H(\kappa)$ such that $\Gamma(u)\mathscr M_{z_j} = (u\cdot\mathscr M_z)_j \Gamma(u)$ for all $j = 1, \ldots, d$, where
\[(u\cdot\mathscr M_z)_j = \sum_{k=1}^d u_{jk} \mathscr M_{z_k}, \quad j = 1, \ldots, d.\]
It follows from \cite[Lemma 2.3]{GKMP} that $\mathscr M_z$ on $\mathcal H(\kappa)$ is $\mathcal U(d)$-homogeneous if and only if $\kappa$ is quasi-invariant under $\mathcal U(d)$ in the sense that for each $u \in \mathcal U(d)$, there exists a unitary operator $J_u$ on $E$ such that $\kappa(z,w) = J_u\, \kappa(u\cdot z, u\cdot w)\, J_u^*$ for all $z,w \in \mathbb B_d$. In the case when $\kappa$ is scalar-valued, it turns out that $\mathscr M_z$ on $\mathcal H(\kappa)$ is $\mathcal U(d)$-homogeneous if and only if $\kappa$ is $\mathcal U(d)$-invariant (see \cite[Theorem 2.12]{CY}). In view of this, note that $\mathscr M_z$ appearing in the Proposition \ref{prop-2} are $\mathcal U(d)$-homogeneous. If $\mathscr M_z$ is $\mathcal U(d)$-homogeneous on $\mathcal H(\kappa)$ with $\kappa$ being a diagonal kernel, then by \cite[Theorem 4.10]{GKMP}, the kernel $\kappa$ has the form
\beqn
\kappa(z,w) = \sum_{m \in \mathbb N} A_m \inp{z}{w}^m, \quad z, w \in \mathbb B_d,
\eeqn
where $\{A_m\}_{m \in \mathbb N}$ is a sequence of positive definite $n \times n$ matrices. 



\begin{proposition}
Let $n$ be a positive integer. For the sequences $\{A_m\}_{m \in \mathbb N}$ and $\{\tilde A_m\}_{m \in \mathbb N}$ of $n \times n$ positive invertible matrices, let $\kappa$ and $\tilde\kappa$ be the reproducing kernels given by  
\beqn
\kappa(z,w) &=& \sum_{m \in \mathbb N} A_m \inp{z}{w}^m, \quad z, w \in \mathbb B_d,\\
\tilde\kappa(z,w) &=& \sum_{m \in \mathbb N} \tilde A_m \inp{z}{w}^m, \quad z, w \in \mathbb B_d.
\eeqn
Let $\mathcal H(\kappa)$ and $\mathcal H(\tilde\kappa)$ be the reproducing kernel Hilbert spaces of $\mathbb C^n$-valued holomorphic functions defined on the open unit ball $\mathbb B_d$ of $\mathbb C^d$, determined by $\kappa$ and $\tilde\kappa$ respectively. Suppose that $\mathscr M_z$ is bounded on $\mathcal H(\kappa)$ and on $\mathcal H(\tilde\kappa)$. Then $\mathscr M_z$ on $\mathcal H(\kappa)$ is similar to $\mathscr M_z$ on $\mathcal H(\tilde\kappa)$ if and only if there exist positive constants $m_1, m_2$ and an invertible $n \times n$ matrix $C$ such that 
\beqn
m_1 C^* A_m C \leqslant \tilde A_m \leqslant m_2 C^* A_m C \quad \mbox{for all } m \in \mathbb N,
\eeqn 
\end{proposition}

\begin{proof}
It is easy to see that 
\beqn
\kappa(z,w) = \sum_{\alpha \in \mathbb N^d} C_\alpha z^\alpha \bar{w}^\alpha \quad  \mbox{ and } \quad \tilde\kappa(z,w) = \sum_{\alpha \in \mathbb N^d} \tilde C_\alpha z^\alpha \bar{w}^\alpha, \ \ z,w \in \mathbb B_d,
\eeqn
where $C_\alpha = \frac{|\alpha|!}{\alpha!} A_{|\alpha|}$ and $\tilde C_\alpha = \frac{|\alpha|!}{\alpha!} \tilde A_{|\alpha|}$ for all $\alpha \in \mathbb N^d$. Since $\mathscr M_z$ on $\mathcal H(\kappa)$ is unitarily equivalent to $\mathscr M_z$ on $\mathcal H^2(\mathscr B)$ and $\mathscr M_z$ on $\mathcal H(\tilde\kappa)$ is unitarily equivalent to $\mathscr M_z$ on $\mathcal H^2(\tilde{\mathscr B})$, where $\mathscr B = \{B_\alpha\}_{\alpha \in \mathbb N^d}$, $\tilde{\mathscr B} = \{\tilde B_\alpha\}_{\alpha \in \mathbb N^d}$ with $B_\alpha^* B_\alpha = C_\alpha^{-1}$ and $\tilde B_\alpha^* \tilde B_\alpha = \tilde C_\alpha^{-1}$ for all $\alpha \in \mathbb N^d$, it now immediately follows from Theorem \ref{similar-thm} that $\mathscr M_z$ on $\mathcal H(\kappa)$ is similar to $\mathscr M_z$ on $\mathcal H(\tilde\kappa)$ if and only if there exist positive constants $m_1, m_2$ and an invertible $n \times n$ matrix $C$ such that 
\beqn
m_1 C^* A_{|\alpha|} C \leqslant \tilde A_{|\alpha|} \leqslant m_2 C^* A_{|\alpha|} C \quad \mbox{for all } \alpha \in \mathbb N^d.
\eeqn 
This completes the proof.
\end{proof}


\begin{thebibliography}{33}
\bibitem{AS} M. Abramowitz and I. A. Stegun, {\it Handbook of mathematical functions with formulas, graphs, and   mathematical tables}, National Bureau of Standards Applied Mathematics Series, 55, U. S. Government Printing Office, Washington, DC, 1964.

\bibitem{CS} R. Curto and N. Salinas, Generalized Bergman kernels and the Cowen-Douglas theory, {\it Amer. J. Math.} {\bf 106} (1984), 447-488.

\bibitem{CY}
S. Chavan, D. Yakubovich, Spherical tuples of Hilbert space operators,
{\it Indiana Univ. Math. J.} {\bf 64} (2015), 577-612.

\bibitem{GKMP}
S. Ghara, S. Kumar, G. Misra and P. Pramanick, Commuting tuple of multiplication operators homogeneous under the unitary group, preprint.

\bibitem{GKT}
R. Gupta, S. Kumar and S. Trivedi, Von Neumann's inequality for commuting operator-valued multishifts, {\it Proc. American Math. Soc.} {\bf 147} (2019), 2599-2608.

\bibitem{GKT1}
R. Gupta, S. Kumar and S. Trivedi, Unitary equivalence of operator-valued multishifts, {\it J. Math. Anal. Appl.} {\bf 487} (2020), 23 pp.

\bibitem{K} 
S. Kumar, Spherically balanced Hilbert spaces of formal power series in several variables-II, {\it Complex Anal. Oper. Theory} {\bf 10} (2016), 505–526.

\bibitem{L} A. Lambert, Unitary equivalence and reducibility of invertibly weighted shifts, {\it Bull. Austral. Math. Soc.} {\bf 5} (1971), 157-173.

\bibitem{PR}  V. Paulsen and M. Raghupathi, {\it An Introduction to the Theory of Reproducing Kernel Hilbert Spaces}, Cambridge Studies in Advanced Mathematics, 152. Cambridge University Press, Cambridge, 2016.

\bibitem{S}
A. Shields, {\it Weighted shift operators and analytic function
theory, in Topics in Operator Theory}, Math. Surveys
Monographs, vol. 13, Amer. math. Soc., Providence, RI 1974, 49-128.


\end{thebibliography}
\end{document}